\pdfoutput=1
\documentclass[reqno,11pt]{amsart}
\usepackage{amssymb}
\usepackage{amsthm}
\usepackage{amsbsy}
\usepackage{amsmath}
\usepackage{mathrsfs}
\usepackage[pdftex]{color}
\usepackage{enumerate,fncylab}
\usepackage{graphicx}
\usepackage[pdftex,pdftitle={Stochastic integrals and conditional full support},pdfauthor={Mikko S. Pakkanen}]{hyperref}

\newcommand{\ud}{\mathrm{d}}
\newcommand{\proba}{\mathbf{P}}
\newcommand{\qrob}{\mathbf{Q}}
\newcommand{\R}{\mathbb{R}}
\newcommand{\Q}{\mathbb{Q}}
\newcommand{\N}{\mathbb{N}}

\newcommand{\eqdefr}{=\mathrel{\mathop:}}
\newcommand{\eqdefl}{\mathrel{\mathop:}=}

\newcommand{\e}{\mathrm{e}}
\newcommand{\ut}{\underline{t}}

\numberwithin{equation}{section}

\theoremstyle{plain}

\newtheorem{thm}[equation]{Theorem}
\newtheorem{lem}[equation]{Lemma}
\newtheorem{cor}[equation]{Corollary}

\theoremstyle{definition}

\newtheorem{rem}[equation]{Remark}

\newtheorem{exmp}[equation]{Example}

\begin{document}

\title[Stochastic integrals and conditional full support]{Stochastic integrals and \\ conditional full support}
\author{Mikko S.\ Pakkanen}
\address{
Department of Mathematics and Statistics \\
University of Helsinki\\
P.O.\ Box 68\\
FI-00014 Helsingin yliopisto \\
Finland}
\email{msp@iki.fi}
\urladdr{http://www.iki.fi/msp/}
\date{\today}

\begin{abstract}
We present conditions that imply the conditional full support (CFS) property, introduced by Guasoni, R\'asonyi, and Schachermayer [Ann.\ Appl.\
  Probab.,\ 18 (2008),\ pp.~491--520], for processes $Z \eqdefl H + K \cdot W$, where $W$ is a Brownian motion, $H$ is a continuous process, and processes $H$ and $K$ are either progressive or independent of $W$. 
Moreover, in the latter case under an additional assumption that $K$ is of finite variation, we present conditions under which $Z$ has CFS also when $W$ is replaced with a general continuous process with CFS. As applications of these results, we show that several stochastic volatility models and the solutions of certain stochastic differential equations have CFS.\end{abstract}

\thanks{Research was supported by the Academy of Finland (project 116747) and the Finnish Cultural Foundation.}

\maketitle

\section{Introduction}

\subsection{Preliminaries}
The purpose of this paper is to show that certain stochastic integrals have the \emph{conditional full support} property, introduced by Guasoni, R\'asonyi, and Schachermayer \cite{guaso2}.
So, before stating the main results of this paper, let us recall the definition of this property. 

To this end, recall first that when $E$ is a separable metric space and $\mu : \mathscr{B}(E) \rightarrow [0,1]$ is a Borel probability measure, the \emph{support} of $\mu$, denoted by $\mathrm{supp}(\mu)$, is the (unique) minimal closed set $A \subset E$ such that $\mu(A)=1$. Let $(X_t)_{t \in [0,T]}$ be a continuous stochastic process taking values in an open interval $I\subset\R$, defined on a complete probability space $(\Omega, \mathscr{F}, \proba)$, and let $\mathbb{F}=(\mathscr{F}_t)_{t \in [0,T]}$ be a filtration on this space. Moreover, let $C_x([u,v],I)$ be the space of functions $f \in C([u,v],I)$ such that $f(u) = x \in I$. As usual, we equip the spaces $C([u,v],I)$ and $C_x([u,v],I)$, $x \in I$ with the uniform topologies.

We say that the process $X$ has \emph{conditional full support} (CFS) with respect to the filtration $\mathbb{F}$, or briefly $\mathbb{F}$-CFS, if 
\begin{enumerate}[(a)]
\item $X$ is adapted to $\mathbb{F}$,
\item for all $t \in [0,T)$ and $\proba$-almost all $\omega \in \Omega$,
\begin{equation}\label{cfsdef}
\mathrm{supp}\big(\mathrm{Law}\big[(X_u)_{u \in [t,T]}\big| \mathscr{F}_{t}\big](\omega)\big) = C_{X_{t}(\omega)}([t,T],I).
\end{equation}
\end{enumerate}
In \eqref{cfsdef}, we regard $\mathrm{Law}\big[(X_u)_{u \in [t,T]}\big| \mathscr{F}_{t}\big]$ as a regular conditional law (a random Borel probability measure, see e.g.\ \cite[pp.\ 106--107]{kalle}) on the space $C([t,T],I)$.
Arguably, the formulation of the CFS property might appear slightly complicated at first sight, but informally it simply means that when we observe $X$ from any time $t\in[0,T)$ onwards,
 $X$ still stays arbitrarily close to any continuous path in $I$ starting from $X_t$ with a positive $\mathscr{F}_t$-conditional probability.
 
Throughout this paper, when we say that some process has CFS without mentioning the filtration, we tacitly take it to be the natural filtration of the process. In fact, this is the ``weakest form'' of CFS that a process can have (see Corollary \ref{smallerfilt} and Lemma \ref{usual}).

\subsection{Main results of the paper} We shall establish CFS for processes of the form
\begin{equation}\label{Zprocess}
Z_t \eqdefl H_t + \int_0^t k_s \ud W_s, \quad t \in [0,T],
\end{equation}
where $H$ is a continuous process, the integrator $W$ is a Brownian motion, and the integrand $k$ satisfies some varying assumptions (to be clarified below). 
\labelformat{enumi}{Z$^{#1}$}

We focus on three cases, each of which requires a separate treatment. First, we study the case:
\begin{enumerate}[(Z$^1$)]
\item\label{caseindep} $H$ and $k$ are (jointly) independent of $W$.
\end{enumerate}
We find that in this case, $Z$ has CFS if the set of zeros of $t \mapsto k_t$ has zero Lebesgue measure almost surely (Theorem \ref{indepcfs}). As an application of this result, we show that several popular \emph{stochastic volatility} models---which feature leverage effects, long memory in volatility, and volatility jumps---have the CFS property. 
Next, we relax the assumption about independence, and consider the case:
\begin{enumerate}[(Z$^1$)]\setcounter{enumi}{1}
\item\label{casedep} $H$ and $k$ depend progressively on $W$ and some additional continuous process that does not anticipate $W$; and $H = \int_0^{\cdot} h_s \ud s$ is absolutely continuous.
\end{enumerate}
It turns out that in this case the simple sufficient condition of the preceeding case does not guarantee CFS---we present a very simple example, in which $k$ is stricly positive but $Z$ fails to have CFS. However, we show that under a set of more stringent assumptions---which are satisfied e.g.\ when $k$ is bounded from above and away from zero, and $h$ is bounded---the process $Z$ does have CFS (Theorem \ref{progcfs}). Using this result, we are able to establish CFS for the solutions of certain stochastic differential equations.
Finally, we consider briefly the following (partial) extension of the case \eqref{caseindep}:
\begin{enumerate}[(Z$^1$)]\setcounter{enumi}{2}
\item\label{casegen} A (general) continuous process $X$ with CFS takes place of $W$; processes $H$ and $k$ are (jointly) independent of $X$; and $k$ is of finite variation.
\end{enumerate}
We show that in this case, $Z$ has CFS if each path of $k$ is bounded away from zero (Theorem \ref{gencfs}).

\subsection{Motivation} 

As mentioned before, the CFS property was first introduced by Guasoni, R\'asonyi, and Schachermayer \cite{guaso2}, in connection to mathematical finance, viz.\ pricing models with \emph{transaction costs}. Their main result asserts that if a continuous price process has CFS, then for any $\varepsilon>0$ there exists a so-called \emph{$\varepsilon$-consistent price system}, which is a martingale (after an equivalent change of measure) that shadows the price process within the bid--ask spread implied by $\varepsilon$-sized proportional transaction costs (see also \cite{kaban} for a related study). The existence of $\varepsilon$-consistent price systems for all $\varepsilon > 0$ implies that the price process does not admit arbitrage opportunities under arbitrary small transaction costs---since any arbitrage strategy would generate arbitrage also in the consistent price system, which is a contradiction because of the martingale property. Consistent price systems can be seen as generalizations of \emph{equivalent martingale measures} (EMM's), since if a price process admits an EMM, then the price process itself qualifies as a trivial $\varepsilon$-consistent price system for any $\varepsilon>0$.   

However, CFS is worth studying even when it comes to price processes that admit EMM's, since it enables the construction of specific consistent price systems that are useful in solving \emph{superreplication} problems under proportional transaction costs. This is manifested by the ``face-lifting'' result in \cite{guaso2}, which says that if $(P_t)_{t \in [0,T]}$ is a price process with CFS, then the superreplication price of any European-style vanilla contingent claim $g(P_T)$ under $\varepsilon$-sized proportional transaction costs tends to $\hat{g}(P_0)$ when $\varepsilon \downarrow 0$, where $\hat{g}$ is the \emph{concave envelope} of $g$ (informally, the ``smallest'' concave function that majorizes $g$). This means that superreplicating e.g.\ a European call option under small proportional transaction costs entails buying the underlying, the trivial superreplicating portfolio.

In addition models with transaction costs, CFS has been found to be have relevance also to frictionless pricing models, as indicated in the recent paper of Bender, Sottinen, and Valkeila \cite{bende}. Their result asserts that if a continuous price process has CFS (they use the name \emph{conditional small ball property}) and \emph{pathwise quadratic variation}, then it does not admit arbitrage opportunities in a class of trading strategies that is somewhat narrower than what the classical models allow, but nevertheless covers a large share (if not almost all) of the strategies that have practical relevance.

Aside from having these applications in mathematical finance, CFS is an interesting fundamental property from a purely mathematical point of view. In particular, research on the CFS property can be seen as a natural continuation to the classical studies of the supports of the laws of continuous Gaussian processes, by Kallianpur \cite{kalli}, and diffusions, initiated by Stroock and Varadhan \cite{stroo} and continued by several other authors (see e.g.\ \cite{mille} and the references therein).

\subsection{Previous results} So far, a variety of continuous processes, which are non-degenerate in some sense, have been shown to have CFS.

Gaussian processes that have CFS include \emph{fractional Brownian motion}, with any Hurst index $\mathfrak{h} \in (0,1)$ (and specifically when $\mathfrak{h}=1/2$, \emph{standard} Brownian motion) \cite{guaso2}, and more generally all \emph{Brownian moving averages} with non-vanishing kernels \cite{chern}. Moreover, Gaussian processes with \emph{stationary increments} that satisfy a certain spectral density  condition have CFS \cite{gasba}.

In the case of continuous Markov processes, showing CFS reduces to showing that the support of the (unconditional) law of the process is the largest possible, as pointed out in \cite{guaso2}. Hence, the classical results concerning the supports of diffusions (e.g.\ \cite{mille,stroo}) can be used to establish CFS.

Moreover, it was shown in \cite{guaso2}, that if continuous process $X$ has CFS, then the Riemann integral process $\int_0^{\cdot}X_t \ud t$ has CFS, which allows (using iteration) the construction of processes that have CFS and arbitrarily smooth paths.

While many examples of common Gaussian and Markovian processes have CFS, this observation should not be extrapolated too much, since CFS is a not a trivial property, due to its \emph{functional} nature. 
For example, non-degeneracy of finite dimensional conditional laws does not typically, even in the case of Gaussian processes, guarantee CFS. A striking example of this difficulty is the continuous Gaussian process $(X_t)_{t \in [0,1]}$ constructed by Cherny \cite[Example 3.1]{chern}, which satisfies
$\mathbf{Var}[X_{t}|\mathscr{F}^X_s]>0$ a.s.\ for all $t,s \in [0,1]$ such that $t>s$, but nevertheless $\int_0^1 X_t \ud t = 1$ almost surely---implying that $X$ cannot have CFS.

\subsection{Outline of the paper} 

Section \ref{basicresults} contains some basic results on the CFS property, to familiarize the reader with the property, and to facilitate subsequent proofs. Section \ref{mainresults} contains the main results concerning stochastic integrals. Finally, Section \ref{applications} concludes with applications to the aforementioned more specific processes: stochastic volatility models and solutions of stochastic differential equations.

\subsection{Notations and conventions}\label{notat} 

Let $\mathbb{T}\subset [0,\infty)$ be a left-closed interval and $(X_t)_{t \in \mathbb{T}}$ a generic stochastic process on $(\Omega,\mathscr{F},\proba)$. We say that $X$ is \emph{jointly measurable}, if $(t,\omega) \mapsto X_t(\omega)$ is measurable with respect to $\mathscr{B}(\mathbb{T})\otimes \mathscr{F}$. Throughout this paper, we assume implicitly that all processes are jointly measurable, whenever this is not already implied by continuity (from left or right) or progressive measurability.
For any $t \in \mathbb{T}$, we write $\widehat{X}^t\eqdefl (X_s - X_t)_{s \geq t}$.

We denote by $\tilde{\mathbb{F}}^X = (\tilde{\mathscr{F}}^X_t)_{t \in \mathbb{T}}$ the ``raw'' natural filtration of $X$ and by $\mathbb{F}^X = (\mathscr{F}^X_t)_{t \in \mathbb{T}}$ its usual augmentation (the minimal right-continuous augmentation of $\tilde{\mathbb{F}}^X$ such that $\mathscr{F}^X_{\min\mathbb{T}}$ contains all $\proba$-null sets in $\tilde{\mathscr{F}}^X_t$ for all $t \in \mathbb{T}$, see e.g.\ \cite[p.\ 45]{revuz}).

As usual, $\| \cdot \|_{\infty}$ denotes the sup-norm, and for any $f,g \in C(\mathbb{T})\eqdefl C(\mathbb{T},\R)$ and $r>0$, write $B(g,r) \eqdefl \{h \in C(\mathbb{T}) : \|h-g\|_{\infty} <r\}$ and $I(f,g,r) \eqdefl \mathbf{1}_{B(g,r)}(f)$.

Finally, $\R_+ \eqdefl (0,\infty)$, $\Q_+ \eqdefl \Q \cap \R_+$, $\N \eqdefl \{ 0,1,\ldots\}$, $\mathbb{Z}_+ \eqdefl \N \setminus \{0\}$, and $\lambda$ stands for the Lebesgue measure on $\R$.

\section{Basic results on the conditional full support property}\label{basicresults}

Since CFS is a very recent concept, in the absence of any comprehensive account, it is instructive to present a few basic results that can be used to establish the property. 
We will consider processes and their CFS in the largest possible state space $\R$, but this is not really a restriction, since all of the following results, except Lemma \ref{smallballprobs}, can be applied also to processes in smaller state spaces using the following observation. 

\begin{rem}
If $I \subset \R$ is an open interval and $f : \R \rightarrow I$ is a homeomorphism, then $g \mapsto f \circ g$ is a homeomorphism between $C_x([0,T])$ and $C_{f(x)}([0,T],I)$. Hence, for $f(X)$, understood as a process in $I$, we have
\begin{equation}\label{gendomain}
\textrm{$f(X)$ has $\mathbb{F}$-CFS} \quad \Leftrightarrow \quad \textrm{$X$ has $\mathbb{F}$-CFS.}
\end{equation}
\end{rem}

We begin with an alternative ``small-ball'' characterization of CFS, which is more tractable than the original definition \eqref{cfsdef}.

\begin{lem}[Small-ball probabilities]\label{smallballprobs} Let $(X_t)_{t \in [0,T]}$ be a continuous process, adapted to filtration $\mathbb{F}=(\mathscr{F}_t)_{t \in [0,T]}$.
Then, $X$ has $\mathbb{F}$-CFS if and only if
\begin{equation}\label{smallballexpect}
\mathbf{E}\big[I\big(\widehat{X}^t,f,\varepsilon\big)\big|\mathscr{F}_t\big] > 0 \quad \textrm{a.s.}
\end{equation}
for all $t \in [0,T)$, $f \in C_0([t,T])$, and $\varepsilon > 0$.
\end{lem}

\begin{proof} Let $t \in [0,T)$ be fixed. For brevity, denote $\mu_{\omega} \eqdefl \mathrm{Law} \big[(X_u)_{u \in [t,T]}  \big| \mathscr{F}_t \big](\omega)$ and  $\hat{\mu}_{\omega} \eqdefl \mathrm{Law} \big[\widehat{X}^t \big| \mathscr{F}_t \big](\omega)$ for all $\omega \in \Omega$.
It is straightforward to check that
\begin{equation}\label{shiftsupport}
\mathrm{supp}(\hat{\mu}_{\omega}) = C_0([t,T]) \quad \Leftrightarrow \quad \mathrm{supp}(\mu_{\omega}) = C_{X_t(\omega)}([t,T]).
\end{equation}
The space $C_0([t,T])$ is separable (e.g.\ by the Stone--Weierstrass theorem), so there exists a countable dense family $\{f_n : n \in \N \} \subset C_0([t,T])$. Hence, the equality on the left hand side of \eqref{shiftsupport} holds if and only if $\hat{\mu}_{\omega}(B(f_n,q))>0$ for all $n \in \N$ and $q \in \Q_+$. By virtue of countability, we find that
$\proba[\mathrm{supp}(\hat{\mu}_{\cdot}) = C_0([t,T])] = 1$ if and only if 
\begin{equation}\label{countability}
\proba[\hat{\mu}_{\cdot}(B(f_n,q))>0] = 1 \quad \textrm{for all $n \in \N$ and $q \in \Q_+$.}
\end{equation}
By the disintegration theorem (Theorem 6.4 of \cite{kalle}), we have $\hat{\mu}_{\cdot}(B(f_n,q)) = \mathbf{E}\big[I\big(\widehat{X}^t,f_n,q\big)\big|\mathscr{F}_t\big]$ a.s., so \eqref{countability} is clearly equivalent to the asserted condition \eqref{smallballexpect}.
\end{proof}

\begin{rem}
While Lemma \ref{smallballprobs} is somewhat obvious, it has two very important consequences. 

Firstly, we note that the characterization \eqref{smallballexpect} is stated in terms of conditional \emph{expectations}. Hence, CFS does not hinge on any particular choice of \emph{versions} of the regular conditional \emph{laws}.

Secondly, whenever we want to argue contrapositively and assume that the CFS property fails to hold, Lemma \ref{smallballprobs} guarantees that there exists a \emph{fixed} ball $B(f,\varepsilon)$ such that $\widehat{X}^t$ exits $B(f,\varepsilon)$ with positive probability. Ignoring separability, the definition of CFS alone would then only imply existence of a \emph{random} balls $B(f(\omega),\varepsilon(\omega))$ with the same property, which would cause certain complications (primarily, the need to find a measurable selection of these balls).
\end{rem}

Thus, establishing CFS reduces to checking that certain conditional expectations are positive.  It is sometimes easier to show positivity of a conditional expectation by arguing that the analogous conditional expectation with respect to some \emph{larger} $\sigma$-algebra is positive, and then pass to the original $\sigma$-algebra using the following elementary fact.

\begin{lem}[Positivity]\label{positivece}
Let $\mathscr{G}$ and $\mathscr{H}$ be $\sigma$-algebras such that $\mathscr{G} \subset \mathscr{H}$, and $Y\in L^1$ such that $Y \geq 0$. If $\mathbf{E}[Y|\mathscr{H}]>0$ a.s., then $\mathbf{E}[Y|\mathscr{G}]>0$ a.s.
\end{lem}

Combining Lemmas \ref{smallballprobs} and \ref{positivece} we find that, like the semimartingale property, CFS is preserved when the filtration is shrinked, as long as the process is adapted to the smaller filtration (this observation is not really new, it was already employed e.g.\ in \cite{bende} and \cite{chern}). 

\begin{cor}[Smaller filtration]\label{smallerfilt} 
Let $(X_t)_{t\in [0,T]}$ be a continuous process, adapt\-ed to filtrations $\mathbb{F}=(\mathscr{F}_t)_{t \in [0,T]}$ and $\mathbb{G}=(\mathscr{G}_t)_{t \in [0,T]}$ that satisfy $\mathscr{G}_t \subset \mathscr{F}_t$ for all $t \in [0,T]$. Then, if $X$ has $\mathbb{F}$-CFS, then it has also $\mathbb{G}$-CFS. 
\end{cor}

Next we shall show a result to the opposite direction, namely that CFS is preserved when the filtration is augmented the usual way (see e.g.\ \cite[p.\ 45]{revuz})

\begin{lem}[Usual augmentation]\label{usual} Let $(X_t)_{t \in [0,T]}$ be a continuous process, adapted to filtration $\mathbb{F}=(\mathscr{F}_t)_{t \in[0,T]}$. Then, $X$ has $\mathbb{F}$-CFS if and only if it has CFS with respect to the usual augmentation of $\mathbb{F}$.
\end{lem}

\begin{proof}
The ``if'' part follows from Corollary \ref{smallerfilt}. Moreover, for the ``only if'' part, it follows from a simple monotone class argument that adding the $\proba$-null sets in $\mathscr{F}_T$ to the filtration does not alter conditional expectations, and hence by Lemma \ref{smallballprobs}, CFS remains intact. Thus, it suffices to show that passing to the right-continuous augmentation $(\mathscr{F}_{t+})_{t \in [0,T]}$ preserves CFS.

To this end, we shall argue contrapositively, that if CFS with respect to $(\mathscr{F}_{t+})_{t \in [0,T]}$ fails at time $\ut$, then CFS with respect to $\mathbb{F}$ must fail at some time $\ut+\varepsilon$, where $\varepsilon>0$ is small. So, when $X$ does not have $(\mathscr{F}_{t+})_{t \in [0,T]}$-CFS, by Lemma \ref{smallballprobs} there exist $\ut \in [0,T)$,  $f \in C_0([\ut,T])$, and $\varepsilon > 0$, such that
$
\proba[A]>0,
$
where $A \eqdefl \left\{ \mathbf{E}\big[I\big(\widehat{X}^{\ut},f,\varepsilon\big) \big| \mathscr{F}_{\ut+} \big] = 0\right\}$. Define
\begin{equation*}
\tau \eqdefl \inf \{ t \in (\ut,T] : |X_t - X_{\ut} - f(t)| \geq \eta \},
\end{equation*}
where $\inf \varnothing \eqdefl \infty$, by convention. Obviously $\mathbf{E}[\mathbf{1}_{\{ \tau < \infty\}} |\mathscr{F}_{\ut+}]=1$ a.e.\ on $A$.
For any $n \in \mathbb{Z}_+$, let us define $A_n \eqdefl \{ \tau > \ut + 1/n \} \cap A$, and note that $A_n \in \mathscr{F}_{\ut + 1/n}$, since $\tau$ is a stopping time with respect to $(\mathscr{F}_{t})_{t \in [\ut,T]}$, and since $A \in \mathscr{F}_{\ut+} \subset \mathscr{F}_{\ut + 1/n}$. Moreover, since $A_n \uparrow A$, we have $\proba[A_m] > 0$ for some $m \in \mathbb{Z}_+$ such that $\ut + 1/m <T$. 
Clearly,
\begin{equation*}
\mathbf{E}[\mathbf{1}_A \mathbf{E}[\mathbf{1}_{\{ \tau < \infty\}}| \mathscr{F}_{\ut + 1/m}]] = \mathbf{E}[\mathbf{1}_A \mathbf{1}_{\{\tau < \infty \}}] = \mathbf{E}[\mathbf{1}_A \mathbf{E}[\mathbf{1}_{\{ \tau < \infty\}}| \mathscr{F}_{\ut +}]]=\proba[A],
\end{equation*}
which implies that $\mathbf{E}[\mathbf{1}_{\{\tau < \infty\}} | \mathscr{F}_{\ut + 1/m}]=1$ a.e.\ on $A$. Further, using inclusions $A_m \subset \{ \tau > \ut+ 1/m \}$ and $A_m \subset A$, we see that
\begin{equation}\label{exitset}\mathbf{E}[\mathbf{1}_{\{ \ut + 1/m < \tau < \infty \}} | \mathscr{F}_{\ut + 1/m}] = \mathbf{E}[\mathbf{1}_{ \{ \tau < \infty\}} | \mathscr{F}_{\ut + 1/m}] = 1 \quad \textrm{ a.e.\ on $A_m$.}
\end{equation}

Now, define $\tilde{f}(t) \eqdefl f(t) - f(\ut+1/m)$, $t \in [\ut+1/m,T]$, $\tilde{\varepsilon}(\omega) \eqdefl \varepsilon - |X_{\ut + 1/m}(\omega)-X_{\ut}(\omega)-f(\ut+1/m)|$ (see Figure \ref{smallertubepic}), and note that, by (\ref{exitset}) we have $\tilde{\varepsilon}>0$ a.e.\ on $A_m$. Moreover,
\begin{equation}\label{smallertube}
\begin{split}
\{\ut +1/m < \tau < \infty \} &= \bigg\{ \sup_{t\in (\ut+1/m,T]}  |X_t - X_{\ut} - f(t)|\geq \varepsilon \bigg\} \\
&\subset \bigg\{ \sup_{t \in (\ut+1/m,T]} |X_t - X_{\ut + 1/m}-\tilde{f}(t)| \geq \tilde{\varepsilon} \bigg\} \eqdefr F
\end{split}
\end{equation}
by the triangle inequality. For all $n \in \mathbb{Z}_+$, define $B_n \eqdefl A_m \cap \{\tilde{\varepsilon} > 1/n \} \in \mathscr{F}_{\ut + 1/m}$. Since $B_n \uparrow A_m \cap \{\tilde{\varepsilon}>0 \}$ and $\proba[ A_m \cap \{\tilde{\varepsilon}>0 \}] = \proba[A_m]>0$, we have $\proba[B_{m'}]>0$ for some $m' \in \mathbb{Z}_+$. Using \eqref{exitset} and \eqref{smallertube}, we find that
\begin{equation*}
\mathbf{E}\big[I\big(\widehat{X}^{\ut + 1/m}, \tilde{f}, 1/m' \big)\big| \mathscr{F}_{\ut + 1/m} \big] \leq 1-\mathbf{E}[\mathbf{1}_{\{\tilde{\varepsilon}>1/m'\}\cap F}|\mathscr{F}_{\ut + 1/m}] = 0 \quad \textrm{a.e.\ on $B_{m'}$}.
\end{equation*}
Thus, by Lemma \ref{smallballprobs}, $X$ does not have $\mathbb{F}$-CFS, which concludes the proof.
\end{proof}

\begin{figure}
\caption{Choosing $\tilde{f}$ and $\tilde{\varepsilon}$ in the proof of Lemma \ref{usual}.}
\label{smallertubepic}
\vspace*{0.3cm}
\includegraphics[scale=0.7]{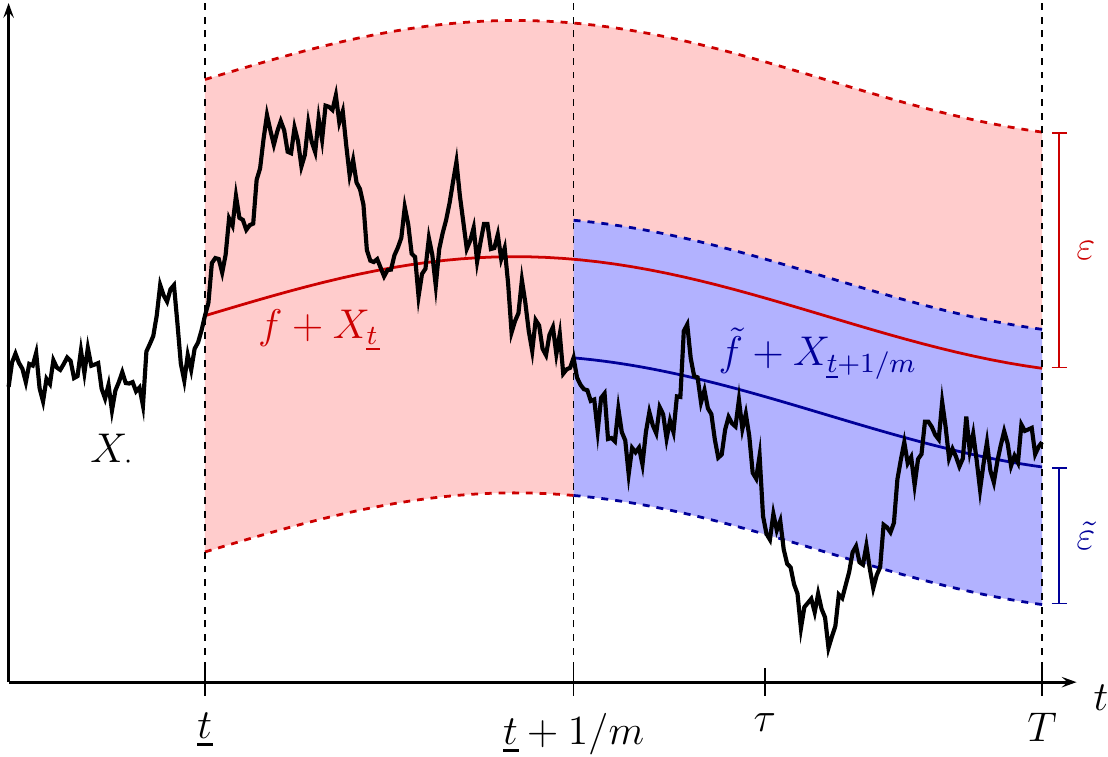}
\end{figure}

We conclude this section by showing that CFS---with respect to the natural filtration of the process---is a property of the law of the process, and thus does not depend on the underlying probability space. This allows flexibility in the subsequent proofs, where it is necessary to assume that the probability space enjoys some specific properties.

\begin{lem}[Law invariance]\label{lawinvariance} Let $(X_t)_{t \in [0,T]}$ and $(Y_t)_{t\in [0,T]}$ be continuous processes (possibly defined on distinct probability spaces) such that $X \stackrel{\mathrm{law}}{=} Y$. Then, $X$ has $\mathbb{F}^X$-CFS if and only if $Y$ has $\mathbb{F}^Y$-CFS.
\end{lem}

\begin{proof}
By Lemma \ref{usual}, it suffices to show the equivalence with respect to ``raw'' natural filtrations. Moreover, it is clearly not a restriction to assume that $X$ and $Y$ are defined on the same probability space.
Let $t \in [0,T)$, $f \in C_0([t,T])$, and $\varepsilon > 0$. Denote by $\{t_k : k \in \N \}$ some enumeration of $[0,t]\cap \Q$. By continuity of paths and Theorem 7.23 of \cite{kalle}, we have $\mathbf{E}\big[I\big(\widehat{X}^t,f,\varepsilon\big)\big|X_{t_1},\ldots X_{t_n}\big] \rightarrow \mathbf{E}\big[I\big(\widehat{X}^t,f,\varepsilon\big)\big|\tilde{\mathscr{F}}^X_t\big]$ and $\mathbf{E}\big[I\big(\widehat{Y}^t,f,\varepsilon\big)\big|Y_{t_1},\ldots Y_{t_n}\big] \rightarrow \mathbf{E}\big[I\big(\widehat{Y}^t,f,\varepsilon\big)\big|\tilde{\mathscr{F}}^Y_t\big]$ a.s.\ when $n \rightarrow \infty$. On the other hand, $X \stackrel{\mathrm{law}}{=} Y$ and Exercise 6.12 of \cite{kalle} imply that $\mathbf{E}\big[I\big(\widehat{X}^t,f,\varepsilon\big)\big|X_{t_1},\ldots X_{t_n}\big] \stackrel{\mathrm{law}}{=} \mathbf{E}\big[I\big(\widehat{Y}^t,f,\varepsilon\big)\big|Y_{t_1},\ldots Y_{t_n}\big]$ for all $n \in \N$. Since an equality in law holds also in the limit, the assertion follows from Lemma \ref{smallballprobs}.
\end{proof}

\section{Conditional full support for stochastic integrals}\label{mainresults}

\subsection{Independent integrands and Brownian integrators}

We shall now move to the main results of this paper, first we establish CFS for the process $Z$, as defined by \eqref{Zprocess}, in the case \eqref{caseindep}.

\begin{thm}[Conditional full support]\label{indepcfs}
Let $(H_t)_{t \in [0,T]}$ be continuous process, $(k_t)_{t \in [0,T]}$ a process such that $\int_0^T k^2_s \ud s < \infty$ a.s., and $(W_t)_{t \in [0,T]}$ a Brownian motion independent of $(H,k)$. 
If \begin{equation}\label{nondegeneracy}
\lambda\big(\{ t \in [0,T] : k_t = 0 \} \big) = 0 \quad \textrm{a.s.,}
\end{equation}
then the process
\begin{equation}
Z_t \eqdefl H_t + \int_0^t k_s \ud W_s, \quad t \in [0,T] 
\end{equation}
has CFS.
\end{thm}

\begin{rem}\label{simplenondeg}
It follows from Fubini's theorem, that if $k_t \neq 0$ a.s.\ for all $t \in [0,T]$, then the condition \eqref{nondegeneracy} holds. Hence, in particular whenever $k_t$ has continuous distribution for all $t$, Theorem \ref{indepcfs} applies.
\end{rem}

\begin{rem}
The process $Z$ does not have to be a semimartingale, as $H$ is only assumed to be continuous. To give a non-trivial example, if
\begin{equation*}
H_t \eqdefl B^{\mathfrak{h}}_t, \quad k_t \eqdefl 1, \quad t \in [0,T],
\end{equation*}
where $B^{\mathfrak{h}}$ is a fractional Brownian motion, independent of $W$, with Hurst index $\mathfrak{h} \in (0,1)$, then $Z$ is a \emph{mixed fractional Brownian motion}, which is not a semimartingale when $\mathfrak{h} \in (0,1/2) \cup (1/2,3/4)$ (as shown by Cheridito \cite{cheri}).   
 \end{rem}

The proof of Theorem \ref{indepcfs} requires some preparation. Specifically, we shall show that the Wiener integral of an almost-everywhere non-vanishing function has positive small-ball probabilities, using a time-change argument similar to the one that appears in \cite{stroo2}.

\begin{lem}[Wiener integrals]\label{wiener}
Let $h \in C([0,T])$, $k \in L^2([0,T])$, $(W_t)_{t \in [0,T]}$ a Brownian motion, and define
\begin{equation*}
J_t \eqdefl h(t) + \int_0^t k(s) \ud W_s, \quad t \in [0,T].
\end{equation*}
If $k(t) \neq 0$ for a.a.\ $t \in [0,T]$,
then for all $\underline{t} \in[0,T)$, $f \in C_0([\underline{t},T])$, and $\varepsilon > 0$ we have
\begin{equation*}
\proba \bigg[ \sup_{t \in [\underline{t},T]} |J_t - J_{\underline{t}} - f(t)| < \varepsilon \bigg] > 0.
\end{equation*}
\end{lem}

\begin{proof}
Clearly, we may assume that $h=0$. Let $\underline{t}\in[0,T)$, $f \in C_0([\underline{t},T])$, and $\varepsilon > 0$. Denote
\begin{equation*}
g(s) \eqdefl \int_{\underline{t}}^t \ud \langle J,J \rangle_u = \int_{\underline{t}}^t k(s)^2 \ud s, \quad t \in [\underline{t},T], 
\end{equation*}
and note that since $k(t) \neq 0$ for a.a.\ $t \in [0,T]$, $g$ is a homeomorphism between $[\underline{t},T]$ and $[0,K]$, where $K \eqdefl \int_{\underline{t}}^T k(s)^2 \ud s$. By the Dambis, Dubins--Schwarz theorem, 
there exists a Brownian motion $(B_s)_{s \in [0,K]}$ such that $J_t-J_{\underline{t}} = B_{g(t)}$, $t \in [\underline{t},T]$ a.s. 
Hence, we obtain
\begin{equation*}
\begin{split}
\sup_{t \in [\underline{t},T]} |J_t-J_{\underline{t}} - f(t)| & = \sup_{t \in [\underline{t},T]} |B_{g(t)}-(f\circ g^{-1})(g(t))| \\
& = \sup_{s \in [0,K]} |B_u - (f \circ g^{-1})(u)| \quad \textrm{a.s.}
\end{split}
\end{equation*}
Since $f \circ g^{-1}$ is continuous, and since the Wiener measure is supported on $C_0([0,K])$ (see e.g.\ Corollary VIII.2.3 of \cite{revuz}), we have
\begin{equation*}
\proba \bigg[ \sup_{t \in [\underline{t},T]} |J_t - J_{\underline{t}} - f(t)| < \varepsilon \bigg] = \proba \bigg[ \sup_{s \in [0,K]} |B_s - (f \circ g^{-1})(s)| < \varepsilon \bigg] > 0. \qedhere
\end{equation*}
\end{proof}

We shall now deduce Theorem \ref{indepcfs} from Lemma \ref{wiener} using a suitable conditioning scheme.

\begin{proof}[Proof of Theorem \ref{indepcfs} (beginning)] Let $\ut \in [0,T)$, $f \in C_0([\ut,T])$, and $\varepsilon > 0$. Further, let $(\Omega,\mathscr{F},\proba)$ be the completed probabilility space that carries $W$, $H$, and $k$. By Lemma \ref{smallballprobs}, it suffices to show that
\begin{equation}\label{smallball}
\mathbf{E} \big[ I\big(\widehat{Z}^{\ut},f,\varepsilon \big) \big| \tilde{\mathscr{F}}^{Z}_{\ut} \big] > 0 \quad \textrm{$\proba$-a.s.,} 
\end{equation}
The proof of this assertion becomes more transparent when we work on an extension of the space $(\Omega,\mathscr{F},\proba)$. Namely, we show an analogous property for a variant of $Z$, denoted by $Z^{\star}$, in which the integrator is $W$ up to time $\ut$, but further Brownian increments of the integrator are defined on an auxiliary space. Then, since $Z$ and $Z^{\star}$ have the same distribution (e.g.\ by Exercise IV.5.16 of \cite{revuz}), it follows that \eqref{smallball} holds, by the argument used in the proof of Lemma \ref{lawinvariance}.

We define the extended space by
\begin{equation*}
\Omega^{\star} \eqdefl \Omega \times C_0([0,T]), \quad \mathscr{F}^{\star} \eqdefl \overline{\mathscr{F}\otimes \mathscr{B}(C_0([0,T]))}, \quad \proba^{\star} \eqdefl \overline{\proba \otimes \nu},
\end{equation*}
where $\nu$ is the Wiener measure on $C_0([0,T])$ and the bars denote completion. For any $\omega^{\star} = (\omega,\omega') \in \Omega^{\star}$, we define $B_t(\omega^{\star})\eqdefl B_t(\omega') \eqdefl \omega'(t)$ and $W^{\star}_t(\omega^{\star}) \eqdefl B_{t \vee \ut}(\omega')-B_{\ut}(\omega')+W_{t\wedge \ut}(\omega)$ for all $t \in [0,T]$. Moreover, we denote by $\mathbf{E}^{\star}$ the expectation with respect to $\proba^{\star}$, by $X$ the identity map on $\Omega$, which can be seen as a random element in the measurable space $(\Omega,\mathscr{F})$, and by $Z^{\star}$ the process analogous to $Z$, with $W^{\star}$ as the integrator. 
Note that by joint measurability, we have $H_t(\omega) = \phi(t,\omega)$ and $k_t(\omega) = \psi(t,\omega)$ for some $\mathscr{B}([0,T])\otimes\mathscr{F}$-measurable functions $\phi$ and $\psi$ from $[0,T]\times \Omega$ to $\R$.
\renewcommand{\qedsymbol}{} 
\end{proof}

For the conclusion of the proof we need the following auxiliary result, which asserts that ``freezing'' randomness on the original probability space $\Omega$ reduces $Z^{\star}$ to a Wiener integral with a drift.

\begin{lem}[Freezing]\label{freezing}
For $\proba$-a.a.\ $\omega \in  \Omega$, we have
\begin{equation}\label{wienerident}
\big(\widehat{Z}^{\star,\ut}_t(\omega,\cdot)\big)_{t\in [\ut,T]} = \bigg(\phi(t,\omega) - \phi(\ut,\omega)+ \int_{\ut}^t \psi(s,\omega) \ud B_s\bigg)_{t \in [\ut,T]}
\end{equation}
up to $\nu$-indistinguishability, where the integral on the right hand side is a Wiener integral.
\end{lem}

\begin{proof} Let us denote by $J^{\omega}$ the process on the right hand side of \eqref{wienerident}. By a standard argument (see e.g.\ Lemma 17.23 of \cite{kalle}), there exist measurable functions $\psi_n$, $n \in \N$ from $[\ut,T] \times \Omega$ to $\R$ such that $\psi_n(t,\omega) = \sum_{i=1}^{k_n} \xi^n_i (\omega)\mathbf{1}_{(t^n_{i},t^n_{i+1}]}(t)$ and $\| \psi_n(\cdot,\omega)- \psi(\cdot,\omega)\|_{L^2[\ut,T]} \rightarrow 0$ when $n \rightarrow \infty$ for $\proba$-a.a.\ $\omega \in \Omega$. Now denote
\begin{equation*}
\begin{split}
I^n_t(\omega,\omega') & \eqdefl \phi(t,\omega) + \phi(\ut,\omega) + \int_{\ut}^t \psi_n(s,\omega) \ud W^{\star}_s(\omega,\omega') \\
& \eqdefl \phi(t,\omega) + \phi(0,\omega) + \sum_{i=1}^{k_n} \xi^n_i(\omega) (B_{t^n_{i+1}\wedge t}(\omega')- B_{t^n_i \wedge t}(\omega')).
\end{split}
\end{equation*}
By $L^2$-continuity of It\^o integrals and Borel--Cantelli lemma, there exists a sequence $n_m \uparrow \infty$ such that $\sup_{t \in [\ut,T]} |\widehat{Z}^{\star,\ut}_t-I^{n_m}_t| \rightarrow 0$ $\proba^{\star}$-a.s.\ when $m \rightarrow \infty$. Hence, by Fubini's theorem
\begin{equation*}
\sup_{t \in [\ut,T]} |\widehat{Z}^{\star,\ut}_t(\omega,\cdot)-I^{n_m}_t(\omega,\cdot)|\rightarrow 0 \quad \textrm{$\nu$-a.s.}
\end{equation*}for $\proba$-a.a.\ $\omega \in \Omega$. On the other hand, $L^2$-continuity of Wiener integrals implies that for $\proba$-a.a.\ $\omega \in \Omega$, also
\begin{equation*}
\sup_{t \in [\ut,T]} |J_t^{\omega}-I^{n_m}_t(\omega,\cdot)|\rightarrow 0 \quad \textrm{in $\nu$-probability.}
\end{equation*}
Now, since we have for $\proba$-a.a.\ $\omega \in \Omega$,
\begin{equation*}
\begin{split}
\mathbf{E}^{\nu}\bigg[\sup_{t \in [\ut,T]}|J^{\omega}_t-\widehat{Z}^{\star,\ut}_t(\omega,\cdot)| \wedge 1\bigg] & \leq \mathbf{E}^{\nu}\bigg[\sup_{t \in [\ut,T]}|\widehat{Z}^{\star,\ut}_t(\omega,\cdot)-I^{n_m}_t(\omega,\cdot)| \wedge 1\bigg]\\
& \quad + \mathbf{E}^{\nu}\bigg[\sup_{t \in [\ut,T]}|J^{\omega}_t-I^{n_m}_t(\omega,\cdot)| \wedge 1\bigg],
\end{split}
\end{equation*}
letting $m \rightarrow \infty$ completes the proof.\end{proof}

\begin{proof}[Proof of Theorem \ref{indepcfs} (conclusion)] Let us denote
$\mathscr{G} \eqdefl \mathscr{F} \otimes \{\varnothing, C_0([0,T]) \}$.
We shall show that
$\mathbf{E}^{\star} \big[I\big(\widehat{Z}^{\star,\ut},f,\varepsilon \big)  \big| \mathscr{G}\big] > 0$ $\proba^{\star}$-a.s., which by Lemma \ref{positivece} implies that the same holds also with respect to $\tilde{\mathscr{F}}^{Z^{\star}}_{\ut} \subset \mathscr{G}$, which in turn implies that \eqref{smallball} holds. We may compose
\begin{equation*}
\widehat{Z}^{\star,\ut}(\omega,\omega') = \widehat{Z}^{\star,\ut}(X(\omega),B(\omega')), \quad (\omega,\omega') \in \Omega^{\star}.
\end{equation*}
Moreover, by independence, $\nu$ is a version of the regular $\mathscr{G}$-conditional law of $B$ on $C_0([0,T])$.
By the disintegration theorem (Theorem 6.4 of \cite{kalle}), we have $\proba^{\star}$-a.s.\
\begin{equation*}
\begin{split}
\mathbf{E}^{\star} \big[ I\big(\widehat{Z}^{\star,\ut},f,\varepsilon \big) \big| \mathscr{G}\big] & = \mathbf{E}^{\star}\big[\mathbf{1}_{B(f,\varepsilon)}\big(\widehat{Z}^{\star,\ut}(X,B)\big)\big|\mathscr{G}\big] \\
& = \int_{C_0([0,T])}\mathbf{1}_{B(f,\varepsilon)}\big(\widehat{Z}^{\star,\ut}(X,\omega')\big) \nu(\ud \omega') \eqdefr Y(X).
\end{split}
\end{equation*}
By Lemma \ref{freezing}, for $\proba$-a.a.\ $\omega \in \Omega$, $\mathbf{1}_{B(f,\varepsilon)}\big(\widehat{Z}^{\star,\ut}(\omega,\cdot)\big) = \mathbf{1}_{B(f,\varepsilon)}\big(J^{\omega}\big)$ $\nu$-a.s., where $J^{\omega}$ is the right hand side of \eqref{wienerident}. But for $\proba$-a.a.\ $\omega \in \Omega$ the map $\psi(\cdot,\omega)$ is a.e.\ non-vanishing, so it follows from Lemma \ref{wiener} that for $\proba$-a.a.\ $\omega\in \Omega$,
\begin{equation*}
Y(X(\omega))=\int_{C_0([0,T])}\mathbf{1}_{B(f,\varepsilon)}(J^{\omega}(\omega')) \nu(\ud \omega')>0.
\end{equation*}
Hence, $Y>0$ also $\proba^{\star}$-a.s., which concludes the proof.
\end{proof}

\subsection{Progressively measurable integrands and Brownian integrators}

In the first case \eqref{caseindep}, $(H,k)$ and $W$ were assumed to be independent. We shall now move to the case \eqref{casedep} and dispense with this assumption. Before stating the result, let us consider two examples that motivate why the conditions we now impose on $H$ and $k$ need to be---apart from allowing dependence---more stringent than earlier.  

\begin{exmp}\label{nocfs1}
When the integrand $k$ is allowed to depend on the Brownian motion $W$, the condition (\ref{nondegeneracy}) is no longer sufficient. Namely, if we set e.g.\ 
\begin{equation*}
H_t \eqdefl 1, \quad k_t \eqdefl e^{W_t - \frac{t}{2}}, \quad t \in [0,T],
\end{equation*}
then $Z = k =\mathscr{E}(W)$, the Dol\'eans exponential of $W$, which is strictly positive and thus does not have CFS, when understood as a process in $\R$.  
\end{exmp}

\begin{exmp}\label{nocfs2}
Even when $k$ is positive and constant, but $H$ depends on $W$, $Z$ might not have CFS. To demonstrate this, let $(B_t)_{t \in [0,T]}$ be a Brownian motion and $\mathbb{G}=(\mathscr{G}_t)_{t \in [0,T]}$ a filtration defined by
\begin{equation*}
\mathscr{G}_t \eqdefl \bigcap_{s > t} (\mathscr{F}^B_s \vee \sigma\{B_T \}), \quad t \in [0,T]. 
\end{equation*}
It is well known (see e.g.\ \cite{jeuli}) that there exists a $\mathbb{G}$-Brownian motion $W$, such that if we define $Z$ with respect to this specific $W$ and set
\begin{equation*}
H_t \eqdefl \int_0^t \frac{B_T - B_s}{T-s}\ud s, \quad k_t \eqdefl 1, \quad t \in [0,T],
\end{equation*}
then $Z=B$, which clearly does not have $\mathbb{G}$-CFS.
\end{exmp}

Intuitively, when $H$ and $k$ depend on $W$, even though $k$ is always positive, they may use ``data'' from $W$ to ''steer'' the process $Z$ away from some regions of the path space, so that CFS does not hold.

\begin{thm}[Conditional full support]\label{progcfs} Let $(X_t)_{t \in [0,T]}$ be a continuous process, $(W_t)_{t \in [0,T]}$ a Brownian motion, $\phi$ and $\psi$ progressive functions from $[0,T] \times C([0,T])^2$ to $\R$, and $\xi$ a random variable. Define
\begin{equation*}
h_t \eqdefl \phi(t,W,X), \ \ k_t \eqdefl \psi(t,W,X), \ \  \mathscr{F}_t \eqdefl \sigma \{\xi, W_s, X_s : s \in [0,t] \}, \ \  t \in [0,T].
\end{equation*}
If $W$ is a Brownian motion also with respect to $\mathbb{F}=(\mathscr{F}_t)_{t \in [0,T]}$,
\begin{equation}\label{hkbounds}
\mathbf{E}\Big[ e^{\lambda \int_0^T k^{-2}_s \ud s}\Big] < \infty \quad \textrm{for all $\lambda>0$},\quad \mathbf{E}\Big[ e^{2 \int_0^T k^{-2}_s h^2_s \ud s}\Big] < \infty, \quad \textrm{and}
\end{equation}
\begin{equation}\label{Kbound}
\int_0^T k^2_s \ud s \leq \overline{K} \quad \textrm{a.s.\ for some constant $\overline{K} \in (0,\infty)$},
\end{equation}
then the process
\begin{equation*}
Z_t \eqdefl \xi + \int_0^t h_s \ud s + \int_0^t k_s \ud W_s, \quad t \in [0,T],
\end{equation*}
has CFS.
\end{thm}

\begin{rem}
In the light of Examples \ref{nocfs1} and \ref{nocfs2}, the condition \eqref{hkbounds} seems nearly optimal. However, it is not that clear how necessary condition \eqref{Kbound} is---i.e., if some non-uniform bound would suffice. That said, for the method we use in the present proof, this uniform bound appears to be unavoidable.
\end{rem}

\begin{proof}
For technical reasons, we assume, without loss of generality, that $W$, $X$, and $\mathbb{F}$ are defined on the whole interval $[0,\infty)$. In order to be able to use regular conditional laws, we define the underlying probability space as follows
\begin{equation*}
\Omega \eqdefl C([0,\infty),\R^2)\times\R, \quad \mathscr{F} \eqdefl \mathscr{B}(C([0,\infty),\R^2)\times \R), \quad \proba \eqdefl \mathrm{Law}[(W,X),\xi],
\end{equation*}
that is, for any $\omega = \big((\omega'_1,\omega'_2),\omega''\big) \in \Omega$ we set $W_t(\omega) \eqdefl \omega'_1(t)$, $X_t(\omega)\eqdefl \omega'_2(t)$ for all $t \in [0,\infty)$, and $\xi(\omega) \eqdefl \omega''$.
Further, let us fix $\underline{t} \in [0,T)$, $f \in C_0 ([\underline{t},\infty))$ such that $f$ is absolutely continuous with bounded Radon--Nikod\'ym derivative $f'$ that satisfies $\mathrm{supp}(f')\subset[\underline{t},T]$, and $\varepsilon>0$. Since $\Omega$ is a Polish space, and by well-known results (see e.g.\ Theorem II.89.1 of \cite{roger}) there exists a regular conditional law $\mu \eqdefl \proba[ \ \cdot \ | \mathscr{F}_{\ut}]$ (a random measure on $(\Omega,\mathscr{F})$).

\emph{Step 1: Conditioning.}
We are about to fix a realization of the random measure $\mu$, drawn from a set having probability one, under which $\widehat{W}^{\ut}$ is a Brownian motion with respect to 
$(\mathscr{F}_t)_{t \in [\ut,\infty)}$.
To this end, for any $d \in \mathbb{Z}_+$, $\mathbf{t} = (t_1,\ldots,t_d) \in [0,\infty)^d$, and $\mathbf{q}=(q_1,\ldots,q_d) \in \R^d$, write
\begin{equation*}
\begin{split}
A_1(\mathbf{t},\mathbf{q}) &\eqdefl \left\{W_{t_1} \leq q_1,\ldots, W_{t_d} \leq q_d \right\}, \\
A_2(\mathbf{t},\mathbf{q}) &\eqdefl \left\{\widehat{W}^{\ut}_{t_1} \leq q_1,\ldots, \widehat{W}^{\ut}_{t_d} \leq q_d \right\}, \quad \textrm{(where $\mathbf{t} \in [\ut,\infty)^d$),} \\
B(\mathbf{t},\mathbf{q}) &\eqdefl \left\{X_{t_1} \leq q_1,\ldots, X_{t_{d-1}} \leq q_{d-1}, \xi \leq q_d \right\}
\end{split}
\end{equation*}
and for any $t \in [\ut,\infty)$,
\begin{equation*}
\begin{split}
\mathscr{I} & \eqdefl \{A_1(\mathbf{t},\mathbf{q}) : d \in \mathbb{Z}_+, \mathbf{t}  \in [0,\ut]^d \cap (\Q \cup \{ \ut \})^d, \mathbf{q} \in \Q^d \}, \\
\mathscr{J}_t & \eqdefl \{A_2(\mathbf{t},\mathbf{q}) : d \in \mathbb{Z}_+, \mathbf{t}  \in [\ut,t]^d \cap (\Q \cup \{ \ut \})^d, \mathbf{q} \in \Q^d \}, \\
\mathscr{K}_t & \eqdefl \{B(\mathbf{t},\mathbf{q}) : d \in \mathbb{Z}_+, \mathbf{t}  \in [0,t]^d \cap (\Q \cup \{ \ut \})^d, \mathbf{q} \in \Q^d \}.
\end{split}
\end{equation*}
Further, we note that
\begin{equation*}
\mathscr{C}_t  \eqdefl \{A_1 \cap A_2 \cap B : A_1 \in \mathscr{I}, A_2 \in \mathscr{J}_t, B \in \mathscr{K}_t \},\quad 
\mathscr{J}_{\infty} \eqdefl \bigcup_{t \in [\ut,\infty)} \mathscr{J}_t
\end{equation*}
are a countable $\pi$-systems, and that by continuity of the associated processes, $\sigma(\mathscr{C}_t) = \mathscr{F}_t$, $t \in [\ut,\infty)$ and $\tilde{\mathscr{F}}^{\widehat{W}^{\ut}}_{\infty}= \sigma(\mathscr{J}_{\infty})$, respectively.
Next, define sets
\begin{equation*}
E \eqdefl \bigcap_{A \in \mathscr{J}_{\infty}} \left\{ \omega \in \Omega : \mu(\omega,A) = \proba[A] \right\},
\end{equation*}
\begin{equation*}
F \eqdefl \bigcap_{\substack{t \in [\ut,\infty) \cap \Q\\s \in \Q_+}}\bigcap_{C \in \mathscr{C}_t}\left\{ \omega \in \Omega : \int_{C} \mu(\omega, \ud \omega') (W_{s+t}(\omega')-W_t(\omega'))=0\right\},
\end{equation*}
\begin{equation*}
G_1 \eqdefl \left\{ \omega \in \Omega : Y_1(\omega) \eqdefl \int_{\Omega} \mu(\omega, \ud \omega') e^{2\| f'\|^2_{\infty} \int_{\underline{t}}^T k^{-2}_s(\omega') \ud s} < \infty \right\},
\end{equation*}
\begin{equation*}
G_2 \eqdefl \left\{ \omega \in \Omega : Y_2(\omega) \eqdefl \int_{\Omega} \mu(\omega, \ud \omega') e^{2 \int_{\underline{t}}^T k^{-2}_s(\omega') h_s^2(\omega') \ud s} < \infty \right\},
\end{equation*}
\begin{equation*}
H \eqdefl \left\{ \omega \in \Omega : \mu \Big( \omega ,\Big\{{\textstyle \int_0^T} k^2_s \ud s \leq \overline{K}\Big\}\Big)=1\right\}.
\end{equation*}
It follows from the assumption that $W$ is a Brownian motion with respect to $\mathbb{F}$, conditions \eqref{hkbounds} and \eqref{Kbound}, and from the disintegration theorem (Theorem 6.4 of \cite{kalle}) that $\proba[E] = \proba[G_1] = \proba[G_2] = \proba[H]=1$. To check that $\proba [F]=1$, note that the intersections in the definition of $F$ are countable, and that
\begin{equation*}
\mathbf{E}[\mathbf{1}_{C}(W_{s+t} - W_t)| \mathscr{F}_{\ut}] = \mathbf{E}[\mathbf{1}_{C} \mathbf{E}[W_{s+t} - W_t | \mathscr{F}_t]| \mathscr{F}_{\ut}] = 0 \quad \textrm{$\proba$-a.s.,}
\end{equation*}
and use the disintegration theorem again.

For the remainder of the proof, we fix $\omega^{\star} \in E \cap F \cap G_1 \cap G_2 \cap H$, denote by $(\mathscr{F}^{\star},\proba^{\star})$ the completion of $(\mathscr{F},\mu(\omega^{\star}, \, \cdot \, ))$, and let $\mathbb{F}^{\star}=(\mathscr{F}^{\star}_t)_{t \in [\underline{t},\infty)}$ be the usual $\proba^{\star}$-augmentation of $(\mathscr{F}_t)_{t \in [\underline{t},\infty)}$. Again, denote by $\mathbf{E}^{\star}$ the expectation with respect to $\proba^{\star}$. Since $\mathscr{J}_{\infty}$ is a $\pi$-system that generates $\tilde{\mathscr{F}}^{\widehat{W}^{\ut}}_{\infty}$, it follows that $W^{\star}\eqdefl\widehat{W}^{\ut}$ is a Brownian motion under $\proba^{\star}$. To complete the first step, we still need to show that $W^{\star}$ is an ($\mathbb{F}^{\star}$, $\proba^{\star}$)-martingale. We note that for any $t \in [\ut,\infty) \cap \Q \cup \{\ut\}$ and $s \in \Q_+$,
\begin{equation*}
\mathbb{H}_{s,t} \eqdefl \left\{ Y \in L^{\infty}(\Omega, \proba^{\star}) : \mathbf{E}^{\star}[Y(W_{s+t} - W_t)] = 0 \right\}
\end{equation*}
is a vector space, closed with respect to uniform convergence, that contains all constant functions and all limits of uniformly bounded increasing sequences of elements of $\mathbb{H}_{s,t}$. Moreover, $\{\mathbf{1}_{C} : C \in \mathscr{C}_t \}$ is closed under multiplication (since $\mathscr{C}_t$ is a $\pi$-system), and since $\omega^{\star} \in F$, we have $\{\mathbf{1}_{C} : C \in \mathscr{C}_t \} \subset \mathbb{H}_{s,t}$. Hence by the functional monotone class theorem (see e.g.\ Th\'eor\`eme I.21 of \cite{della}), $\mathbb{H}_{s,t}$ contains all bounded $\mathscr{F}_t$-measurable random variables. Thus, we find that $\mathbf{E}^{\star}\big[W^{\star}_{s+t}\big|\mathscr{F}_t\big] = W^{\star}_t$ for any $t\in [\ut,\infty) \cap \Q$ and $s \in \Q_+$. Now, let $t',t'' \in [\ut,\infty)$ be such that $t''>t'$, and let $(t''_n)\subset [t,t'') \cap \Q$ and $(t'_n) \subset [\ut,t') \cap \Q \cup \{ \ut \}$ be such that $t''_n \uparrow t''$ and $t'_n \uparrow t'$ respectively when $n \uparrow \infty$. Since by continuity of $X$ and $W$,
\begin{equation*}
\mathscr{F}_{t'} = \bigvee_{m=1}^{\infty} \mathscr{F}_{t'_m},
\end{equation*}
we have by Theorem 7.23 of \cite{kalle} and continuity of $W^{\star}$,
\begin{equation*}
\mathbf{E}^{\star}[W^{\star}_{t''_n}|\mathscr{F}_{t'}] = \lim_{m \rightarrow \infty} \mathbf{E}^{\star}[W^{\star}_{t''_n}|\mathscr{F}_{t'_m}] = \lim_{m \rightarrow \infty} W^{\star}_{t'_m} = W^{\star}_{t'} \quad \textrm{$\proba^{\star}$-a.s.}
\end{equation*}
Moreover, since $\mathbf{E}^{\star}[\sup_{n \in \N}|W^{\star}_{t''_n}|]<\infty$, the dominated convergence theorem for conditional expectations implies
\begin{equation*}
\mathbf{E}^{\star}[W^{\star}_{t''}|\mathscr{F}_{t'}] = \lim_{n \rightarrow \infty} \mathbf{E}^{\star}[W^{\star}_{t''_n}|\mathscr{F}_{t'}] = W^{\star}_{t'} \quad \textrm{$\proba^{\star}$-a.s.}
\end{equation*}
It follows now from standard results (e.g.\ Lemma II.72.2 of \cite{roger}) that $W^{\star}$ is an ($\mathbb{F}^{\star}$, $\proba^{\star}$)-martingale.

\emph{Step 2: Characterization of the support.}
To complete the proof, we shall now show that the law of $\widehat{Z}^{\ut}$ under the measure $\proba^{\star}$ is supported on $C_0([\ut,T])$. To this end, we adapt the method employed in the proof of Lemma 3.1 of \cite{stroo}. For the convenience of the reader, we do this in detail. 
Let us define
\begin{equation*}
Z^{\star}_t \eqdefl \int_{\underline{t}}^t \hat{k}_{s} \ud W^{\star}_s - \left\langle L, \int_{\underline{t}}^{\cdot} \hat{k}_{s} \ud W^{\star}_s \right\rangle_t, \quad t \in [\underline{t},\infty),
\end{equation*}
where
\begin{equation*}
\hat{k}_t \eqdefl k_t \mathbf{1}_{[\underline{t},T]}(t) + \mathbf{1}_{(T,\infty)}(t), \quad L_t \eqdefl \int_{\underline{t}}^t \hat{k}_s^{-1}(f'(s) - h_s\mathbf{1}_{[\underline{t},T]}(s)) \ud W^{\star}_s.
\end{equation*}
We can easily check that 
\begin{equation}\label{zidentity}
Z^{\star}_t = Z_t - Z_{\underline{t}}-f(t),\quad t \in [\underline{t},T].
\end{equation}
Further, we have for all $t \in [\underline{t},\infty)$ the uniform bound
\begin{equation*}
\begin{split}
\langle L, L \rangle_t  & = \int_{\underline{t}}^t \hat{k}_s^{-2}(f'(s) - h_s\mathbf{1}_{[\underline{t},T]}(s))^2 \ud s \\
& \leq 2 \left( \| f'\|^2_{\infty} \int_{\underline{t}}^T k^{-2}_s \ud s + \int_{\underline{t}}^T k^{-2}_s h^2_s \ud s\right),
 \end{split}
\end{equation*}
using which we obtain by the Cauchy--Schwarz inequality,
\begin{equation*}
\mathbf{E}^{\star}\Big[e^{\frac{1}{2} \langle L , L\rangle_{\infty}}\Big] \leq \sqrt{Y_1(\omega^{\star})Y_2 (\omega^{\star})} < \infty.
\end{equation*}
Consequently, Novikov's criterion implies that the Dol\'eans exponential $\mathscr{E}(L)_t \eqdefl e^{L_t - \frac{1}{2}\langle L, L \rangle_t}$, $t \in [\underline{t},\infty)$ is a uniformly integrable $(\mathbb{F}^{\star},\proba^{\star})$-martingale, closable at $\infty$. Hence, we may define a new probability measure $\qrob^{\star} \sim \proba^{\star}$ by $\qrob^{\star} \eqdefl \mathscr{E}(L)_{\infty} \cdot \proba^{\star}$.

Now, since $\int_{\underline{t}}^{\cdot} \hat{k}_{s} \ud W^{\star}_s$ is an $(\mathbb{F}^{\star},\proba^{\star}$)-martingale, by Girsanov's theorem $Z^{\star}$ is an $(\mathbb{F}^{\star},\qrob^{\star}$)-martingale. 
Moreover, we have
\begin{equation*}
\langle Z^{\star}, Z^{\star} \rangle_t = \int_{\underline{t}}^t \hat{k}^2_{s} \ud s \geq (t-T)_+ \xrightarrow[t \rightarrow \infty]{} \infty,
\end{equation*}
so by the Dambis, Dubins--Schwarz theorem, there exists a Brownian motion $(B_t)_{t \in [0,\infty)}$ under $\qrob^{\star}$, such that $Z^{\star}_t = B_{ \langle Z^{\star}, Z^{\star} \rangle_t}$, $t \in [\underline{t},\infty)$ $\qrob^{\star}$-a.s. Thus, we have
\begin{equation}\label{precfs}
\begin{split}
\qrob^{\star} \bigg[ \sup_{t \in [\underline{t},T]}|Z^{\star}_t| < \varepsilon \bigg] & = \qrob^{\star} \bigg[ \sup_{t \in [\underline{t},T]}|B_{\langle Z^{\star}, Z^{\star} \rangle_t}| < \varepsilon \bigg] \\
& \geq \qrob^{\star} \bigg[ \sup_{u \in [0,\overline{K}]}|B_u| < \varepsilon \bigg] > 0,
\end{split}
\end{equation}
since $\langle Z^{\star}, Z^{\star} \rangle_t \leq \overline{K}$ for all $t \in [\underline{t},T]$, and since the Wiener measure is supported on $C_0([0,\overline{K}])$ (see e.g.\ Corollary VIII.2.3 of \cite{revuz}). By the equivalence of the measures, we may substitute $\qrob^{\star}$ for $\proba^{\star}$ in (\ref{precfs}). Using the fact that $\proba^{\star}$ coincides with $\mu(\omega^{\star}, \, \cdot \, )$ on $\mathscr{F}$ and
\begin{equation*}
B' \eqdefl \bigg\{ \sup_{t \in [\underline{t},T]}|Z_t - Z_{\underline{t}} - f(t)| < \varepsilon \bigg\} \in \mathscr{F},
\end{equation*}
by (\ref{zidentity}) we have $\mu(\omega^{\star},B')>0$.
To conclude, note that functions having the properties of $f$ are dense in $C_0([\underline{t},T])$, so $Z$ has $\mathbb{F}$-CFS by Lemma \ref{smallballprobs}. Finally, $Z$ has $\mathbb{F}^Z$-CFS by Corollary \ref{smallerfilt} and Lemma \ref{usual}.
\end{proof}

\subsection{Independent integrands and general integrators}

Since Brownian motion has CFS, one might wonder if the preceeding results generalize to the case where the integrator is merely a continuous process with CFS. While the proofs of these results use quite heavily methods specific to Brownian motion (martingales, time changes), in the case independent integrands of finite variation \eqref{casegen} we are able to prove this conjecture.

\begin{thm}[Conditional full support]\label{gencfs}
Let $(H_t)_{t \in [0,T]}$ be a continuous process, $(k_t)_{t \in [0,T]}$ a process of finite variation, and $(X_t)_{t\in[0,T]}$ a continuous process independent of $(H,k)$. If $X$ has CFS and
\begin{equation}\label{nonsingularity}
\inf_{t\in[0,T]} |k_t| > 0 \quad \textrm{a.s.,} 
\end{equation}
then the process
\begin{equation}\label{genint}
Z_t \eqdefl H_t + \int_0^t k_s \ud X_s, \quad t \in [0,T]
\end{equation}
has CFS.
\end{thm}

\begin{rem}
The stochastic integral in \eqref{genint} exists as a \emph{pathwise} Riemann--Stieltjes integral. This well-known fact follows from the (discrete) integration-by-parts formula.
\end{rem}

\begin{proof}
First, denote by $\mathbb{G} = (\mathscr{G}_t)_{t \in [0,T]}$ the filtration given by $\mathscr{G}_{t} \eqdefl \tilde{\mathscr{F}}^X_{t} \vee \sigma \{H_s, k_s : s \in [0,T]\}$, and let $\underline{t} \in [0,T)$, $f \in C_0([\underline{t},T])$, and $\varepsilon > 0$.  Further, define 
\begin{equation*}
g_t \eqdefl \int_0^t k^{-1}_s \ud (f(s)+H_{\underline{t}} - H_s), \quad t \in [\ut,T],
\end{equation*}
which is well defined since $(k^{-1}_t)_{t \in [\ut,T]}$ is of finite variation, by \eqref{nonsingularity}. 
Integration by parts and the Love--Young inequality yield for all $t \in [\underline{t},T]$,
\begin{equation*}
\begin{split}
|Z_t - Z_{\underline{t}}-f(t)| & = \left|\int_{\underline{t}}^t k_s \ud (X_s - X_{\underline{t}}-g_s) \right|\\
 & = \left|k_t (X_t - X_{\underline{t}}-g_t) - \int_{\underline{t}}^t (X_s - X_{\underline{t}}- g_s) \ud k_s \right| \\
&  \leq \underbrace{\bigg(\sup_{t \in [\underline{t},T]} |k_s| + \mathrm{TV}_{[\underline{t},T]}(k)\bigg)}_{\eqdefr M(k)} \cdot \sup_{t \in [\underline{t},T]} |X_t - X_{\underline{t}} - g_t|,
\end{split}
\end{equation*}
where $\mathrm{TV}_{[\underline{t},T]}(k)$ denotes the total variation of the path of $k$ on the interval $[\underline{t},T]$. This estimate implies the inclusion
\begin{equation*}
\bigg\{\sup_{t \in [\underline{t},T]} |X_t - X_{\underline{t}}-g_t| < \frac{\varepsilon}{2 M(k)}  \bigg\} \subset \bigg\{ \sup_{t \in [\underline{t},T]}|Z_t - Z_{\underline{t}}-f(t)| < \varepsilon \bigg\}.
\end{equation*}
Hence, by monotonicity of conditional expectations, and by Lemmas \ref{smallballprobs}, \ref{positivece}, and \ref{usual}, it suffices to show that $\mathbf{E}\big[I\big(\widehat{X}^{\ut}, g, \varepsilon / (2 M(k))\big) \big| \mathscr{G}_{\underline{t}}\big]>0$ a.s. To this end, note that since $X$ is independent of $H$ and $k$, it has $\mathbb{G}$-CFS, and denote by $\nu$ some regular $\mathscr{G}_{\underline{t}}$-conditional law of $\widehat{X}^{\ut}$. The disintegration theorem (Theorem 6.4 of \cite{kalle}) yields now
\begin{equation*}
\mathbf{E}\big[I\big(\widehat{X}^{\ut}, g, \varepsilon / (2 M(k))\big) \big| \mathscr{G}_{\underline{t}}\big] = \int_{C([\underline{t},T])} \mathbf{1}_{B(g, \varepsilon/(2M(k)))}(x) \nu( \, \cdot \,, \ud x) > 0 \quad \textrm{a.s.,}
\end{equation*}
since $X$ has $\mathbb{G}$-CFS.\end{proof}

\section{Applications}\label{applications}

In this section, we establish CFS for certain price processes used in mathematical finance. The main motivation for these applications stems from the desire to uncover more \emph{concrete} price processes to which the  superreplication result in \cite{guaso2} applies.

\subsection{Stochastic volatility models} 

Let us consider a price process $(P_t)_{t \in [0,T]}$ in $\R_+$ defined by
\begin{equation}\label{genstochvol}
\ud P_t = P_t \big(f(t,V_t) \ud t +   \rho g(t,V_t) \ud B_t+\sqrt{1- \rho^2} g(t,V_t) \ud W_t\big), \quad P_0 = p_0 \in \R_+, 
\end{equation}
where $f,g \in C([0,T]\times\R^d)$, $\rho \in (-1,1)$, $(W,B)$ is a planar Brownian motion, $V$ is a process in $\R^d$ such that $g(t,V_t) > 0$ a.s.\ for all $t \in [0,T]$. Further, we assume that $(B,V)$ is independent of $W$. We may now verify that $P$ has CFS, since by positivity of $P$, It\^o's formula yields
\begin{equation*}
\begin{split}
\log P_t & = \log p_0 + \int_0^t \bigg(f(s,V_s) - \frac{1}{2} g(s,V_s)^2\bigg) \ud s + \rho \int_0^t g(s,V_s) \ud B_s \\
& \quad +\sqrt{1- \rho^2} \int_0^t g(s,V_s) \ud W_s,
\end{split}
\end{equation*}
which clearly satisfies the assumptions of Theorem \ref{indepcfs} (See Remark \ref{simplenondeg}), and finally we may invoke \eqref{gendomain}.

Let us review briefly some well-known special cases of \eqref{genstochvol}. In standard Markovian stochastic volatility models, the process $V$ is a one-dimensional diffusion driven by $B$, that is
\begin{equation*}
\ud V_t = \alpha(t,V_t) \ud t + \beta(t,V_t) \ud B_t, \quad V_0 = v_0 \in \R.
\end{equation*}
The popular models introduced by Heston (with leverage effects when $-1<\rho<0$); Hull and White; E.\ M.\ Stein, J.\ C.\ Stein and Scott; and Wiggins are special cases of \eqref{genstochvol}, see e.g.\ \cite{frey} for details. 
One notable special case of \eqref{genstochvol}, in which $V$ is not Markovian, is the model of Comte and Renault \cite{comte}, which was designed to capture long-memory effects in volatility. It can be obtained by setting $g(t,v) \eqdefl \e^v$, $\rho \eqdefl 0$, and choosing $V$ to be a fractional Ornstein--Uhlenbeck process independent of $W$. 

Special cases of \eqref{genstochvol} in which volatility may jump include the stochastic volatility model of Barndorff--Nielsen and Shephard \cite{barnd} and the regime switching model of Guo \cite{guo}. To see why in the former model we have $g(t,V_t) > 0$ a.s.\ for all $t \in [0,T]$, recall that we obtain the model by specifying
\begin{equation}\label{gou}
V_t \eqdefl \int_{-\infty}^t \e^{-\lambda (t-s)}  \ud L_{\lambda s}, \quad t \in [0,T],
\end{equation}
where $(L_t)_{t \in \R}$ is an increasing L\'evy process (a \emph{subordinator}) without drift, such that $L_0 = 0$ a.s., independent of $W$, $\rho\eqdefl0$, and $g(t,v)\eqdefl\sqrt{v}$. To exclude the uninteresting case with zero volatility (when $P$ of course does not have CFS), let us assume that $\proba[L_1 = 0]<1$. We find that $V_t \geq \e^{-\lambda T} V_0$ for all $t \in [0,T]$
since $L$ is increasing. Further, using stationarity and independence of the increments of $L$, we obtain for all $k \in \mathbb{Z}_{+}$,
\begin{equation*}
\proba[V_0 = 0] \leq \proba[L_{-i+1}-L_{-i}=0, \ \forall i = 1,\ldots,k] = \proba[L_1= 0]^k,
\end{equation*}
and letting $k \rightarrow \infty$ we find that $V_0>0$ a.s., from which the desired property follows.

\begin{rem}\label{abscont}
The absolute continuity of the drift of $P$ is of course not necessary. Namely, we can easily establish CFS similarly e.g.\ for any price process $P_t \eqdefl e^{f(t) + (g(V)\cdot W)_t}$, $t \in [0,T]$, where $f$ is an arbitrary continuous function. For certain choices of $f$, process $P$ is known to admit arbitrage opportunities in frictionless pricing models (see \cite{delba}). 
\end{rem}

\subsection{Stochastic differential equations}

Finally, let us consider a price process $(P_t)_{t \in [0,T]}$ in $\R_+$ given by stochastic differential equation
\begin{equation}\label{pricesde}
\ud P_t = \mu(t,P) \ud t + \sigma(t, P) \ud W_t, \quad P_0 = p_0 \in \R_+,
\end{equation}
where $\mu$ and $\sigma$ are progressive functions such that for some constants $\overline{\mu} > 0$ and $\overline{\sigma}>1$, 
\begin{equation}\label{musigmabounds}
|\mu(t,x)| \leq \overline{\mu} x(t), \quad \overline{\sigma}^{-1} x(t) \leq |\sigma(t,x)| \leq \overline{\sigma} x(t), \quad x \in C_{p_0}([0,T],\R_+), \, t \in [0,T].
\end{equation}
Further, we assume $\mu$ and $\sigma$ are such that \eqref{pricesde} has a weak solution, which by definition means that there exists a filtration $\mathbb{F}=(\mathscr{F}_t)_{t \in [0,T]}$ and $\mathbb{F}$-adapted continuous processes $P$ and $W$ that solve \eqref{pricesde}, such that $W$ is an $\mathbb{F}$-Brownian motion. Clearly, we may assume that $\mathscr{F}_t = \sigma \{P_s,W_s : s \in [0,t] \}$. Now, since $P$ is positive, It\^o's formula yields
\begin{equation*}
\log P_t = \log p_0 + \int_0^t \bigg(\frac{\mu(s,P)}{P_s} - \frac{\sigma^2(s,P)}{2P^2_s} \bigg) \ud s + \int_0^T \frac{\sigma(s,P)}{P_s} \ud W_s. 
\end{equation*}
Setting $X \eqdefl P$ and noting that \eqref{musigmabounds} implies the conditions of Theorem \ref{progcfs}, by \eqref{gendomain} it follows that $P$ has CFS.

\section*{Acknowledgements}

I would like to thank Tommi Sottinen for introducing me to the topic and for helpful discussions. I also thank Esa Nummelin for reading previous drafts of this paper and for constant encouragement. Finally, I thank Boualem Djehiche, Dario Gasbarra, and Esko Valkeila for valuable comments.

\end{document}